\documentclass{article}
\usepackage[utf8]{inputenc}
\usepackage{amsmath,amsthm,amssymb,amsfonts}
\usepackage{hyperref,color}
\usepackage{graphicx}
\usepackage[title]{appendix}

\newcommand{\bN}{ {\mathbb  N}}

\newcommand{\bR}{ {\mathbb R}}

\newcommand{\ie}{{\it i.e.}}

\def\a{\mathbf{a}}
\def\b{\mathbf{b}}

\newtheorem{theorem}{Theorem}[section]
\newtheorem{definition}[theorem]{Definition}
\newtheorem{lemma}[theorem]{Lemma}
\newtheorem{corollary}[theorem]{Corollary}
\newtheorem{remark}[theorem]{Remark}
\newtheorem{conjecture}[theorem]{Conjecture}
\newtheorem{proposition}[theorem]{Proposition}

\title{Log-concavity and log-convexity of series containing multiple Pochhammer symbols
\footnotetext{\textit{\textnormal{2020} AMS Mathematics Subject Classification: 26A51, 33C20, 33C60, 33F10}
\\ \textit{Key words and phrases: logarithmic concavity, logarithmic convexity, Tur\'{a}nian, rising factorial, Pochhammer symbol, hypergeometric functions, Fox-Wright function}
}}

\author{Dmitrii Karp \thanks{Holon Institute of Technology, Holon, Israel. Email: dimkrp@gmail.com}, Yi Zhang \thanks{Corresponding author. Department of Foundational  Mathematics, School of Mathematics and Physics, Xi'an Jiaotong-Liverpool University, 
 Suzhou, 215123, China.  Email: Yi.Zhang03@xjtlu.edu.cn}}
\date{\today}

\begin{document}

\maketitle

\begin{abstract}
In this paper, we study power series with coefficients equal to a product of a generic sequence and an explicitly given function of a positive parameter expressible in terms of the Pochhammer symbols. Four types of such series are treated. We show that logarithmic concavity (convexity) of the generic sequence leads to logarithmic concavity (convexity) of the sum of the series with respect to the argument of the explicitly given function.  The logarithmic concavity (convexity) is derived from a stronger property, namely, positivity (negativity) of the power series coefficients of the so-called generalized Tur\'{a}nian.   Applications to special functions such as the generalized hypergeometric function and the Fox-Wright function are also discussed. 
\end{abstract}

\section{Introduction}
Let $\bN$ be the set of positive integers and  $\bN_0 := \bN \cup \{0\}$, $\bR_{+}=[0,\infty)$.
Define a formal power series 
\begin{equation}\label{eq:f-defined}
f(\mu; x) = \sum_{n = 0}^\infty f_n\phi_n(\mu)\frac{x^n}{n!}
\end{equation}
with non-negative coefficients $f_n\phi_n(\mu)$ which depend continuously on a non-negative parameter $\mu$. Our main focus in this paper is the logarithmic concavity (convexity) of the function $\mu\to f(\mu;x)$, \ie, concavity (convexity) of $\mu\to\log(f(\mu;x))$ for a fixed $x$ in the convergence domain of the series  \eqref{eq:f-defined}.  To this end, we define the so-called ``generalized Tur\'{a}nian" for any $\alpha,\beta\ge0$  by the expression
\begin{equation}\label{eq:genTur-defined}
\Delta_f(\alpha,\beta;x)=f(\mu+\alpha;x)f(\mu+\beta;x)-f(\mu;x)f(\mu+\alpha+\beta;x)=\sum_{k=0}^{\infty}\delta_kx^k.
\end{equation}
It is well-known and is easy to see that the condition $\Delta_f(\alpha,\beta;x)\ge0$  ($\le0$) implies log-concavity (log-convexity) of $\mu\to f(\mu;x)$.  It is less trivial that for continuous functions the reverse implication also holds \cite{Ng1987}.  In this paper we will mostly deal with the stronger property: if the coefficients $\delta_k$ at all powers of $x$ in \eqref{eq:genTur-defined} are non-negative (non-positive) for all $\alpha,\beta\ge0$ we will say that $\mu\to f(\mu;x)$ is {\em coefficient-wise log-concave \emph{(}log-convex\emph{)}}.  If this property holds for $(\alpha,\beta)\in A$ for some subset $A$ of $\bR_{+}\times\bR_{+}$, we say that  $\mu\to f(\mu;x)$ is {\em coefficient-wise log-concave \emph{(}log-convex\emph{)} for shifts $\alpha,\beta$ in $A$}.  Note that coefficient-wise log-concavity/log-convexity is well-defined regardless of convergence or divergence of the series  \eqref{eq:f-defined}.

In our previous work \cite{KK2013,KKJMAA2013,KSJMAA2010} we considered a number of explicitly specified functions  $\phi_n(\mu)$, while a non-negative sequence $f_n$ remained generic.  In many cases this sequence will be required to satisfy one of the following properties.   
\begin{definition} \label{DEF:doublypositiveseq}
Let $\{f_k\}_{k = 0}^\infty$ be a non-negative nontrivial real sequence. We call $\{f_k\}_{k = 0}^\infty$ log-concave \emph{(}or  P\'{o}lya frequency two -- $PF_2$, or doubly positive\emph{)} if $f_k^2 \ge f_{k-1}f_{k+1}$ for each $k \in \bN$, and $\{f_k\}_{k=0}^\infty$ is supported on an interval of integers \emph{(}\ie, $f_N = 0$ implies either $f_{N+i} = 0$ for all $i \in \bN_0$ or $f_{N-i} = 0$ for $i=0,\ldots,N$\emph{)}. If $f_k^2\le f_{k-1}f_{k+1}$ for all $k \in \bN$, the sequence is called log-convex \emph{(}this inequality implies that a nontrivial sequence is strictly positive\emph{)}.
\end{definition}
Denote by $(\mu)_n$ the rising factorial, $(\mu)_0=1$, $(\mu)_{n}=\mu(\mu+1)\cdots(\mu+n-1)$.  In \cite{KSJMAA2010}, S.\:M.\:Sitnik and the first author  proved the following result: the function 
\begin{equation}\label{eq:KS2010}
\mu\to f(\mu; x)= \sum_{n=0}^{\infty}f_n \frac{(\mu)_{n}}{n!}x^n
\end{equation}
is coefficient-wise log-concave if $\{f_n\}_{n\ge0}$ is log-concave and coefficient-wise log-convex if $\{f_n\}_{n\ge0}$ is log-convex. This kind of duality is made possible by log-neutrality of the function $f(\mu;x)$ for $f_n\equiv C>0$ which is both log-concave and log-convex sequence (and it is the only sequence satisfying both these properties simultaneously).  Indeed, in this case $f(\mu;x)=C(1-x)^{-\mu}$, so that  $\Delta_f(\alpha,\beta;x)=0$ for all $\alpha,\beta$. The first goal of this paper is to investigate a generalization of this result to the case when $(\mu)_{n}/n!$ in \eqref{eq:KS2010} is replaced by $(\mu)_{nr}/(nr)!$ for $r=2,3,\ldots$.  Setting  $f_n\equiv C>0$ in this case does not lead to a log-neutral function, so the above kind of duality is not possible here. Hence, instead of one series \eqref{eq:KS2010} we will consider two different series depending on whether the sequence $\{f_n\}_{n\ge0}$ is assumed to be log-concave or log-convex.   The second, modified series was suggested by Ahn Ninh (private communication, 2017) by replacing $(\mu)_{n}/n!$ in \eqref{eq:KS2010} by $(\mu)_{nr}/(nr-1)!$ and starting the summation from $n=1$.  These two types of series are discussed in Section~\ref{SEC:mps}.  

Nevertheless, there exist other cases when setting $f_n\equiv C>0$ does lead to a log-neutral function. In particular, taking $\phi_n(\mu)=(\mu)_{2n}/(\mu+1)_{n}$ in \eqref{eq:f-defined} leads to such series. However, in this case we only managed to prove coefficient-wise log-concavity (log-convexity) of $\mu\to f(\mu;x)$ for suitably restricted shifts. The corresponding results and conjectures are discussed in Section~\ref{SEC:sps}. 

In Section~\ref{SEC:tps}, we deal with the series of the form \eqref{eq:f-defined} with $\phi_n(\mu)=(\mu)_{n}/(2\mu)_{n}$. We demonstrate that it is coefficient-wise log-convex for each non-negative sequence $\{f_n\}_{n\ge0}$ and conjecture coefficient-wise log-concavity when $\phi_n(\mu)$ is replaced by $[\phi_n(\mu)]^{-1}$.

In Section~\ref{SEC:app}, we illustrate our results with several applications. We show that they are well-suited for many special functions playing an important role in fractional calculus \cite{VK2020}. First, we remark that all our  claims can be immediately generalized to the series containing the so-called $k$-shifted factorials and $k$-Gamma functions~\cite{DP2007}. Next, new log-convexity/concavity statements are presented for the generalized hypergeometric functions and their parameter derivatives. Last but not least, we furnish similar statements for the Fox-Wright function.

\section{Preliminaries} \label{SEC:pre}

In this section, we present several lemmas which will serve as our main tools in the subsequent investigation. 

\begin{lemma}~\cite[Lemma 5]{KKJMAA2013} \label{LEM:positivity}
Suppose $u, v, r, s> 0$, $u = \max(u, v, r, s)$ and $uv > rs$. Then $u + v > r + s$.
\end{lemma}

The following lemma is a slight generalization of~\cite[Lemma 6]{KKJMAA2013}. We say that a sequence has no more than one change of sign if it has the pattern $(--\cdots--00\cdots00++\cdots++)$, where any of the three parts may be missing.  

\begin{lemma}~\cite[Lemma 6]{KKJMAA2013} \label{LEM:keylemma}
Suppose that $\{ f_k \}_{k = 0}^\infty$ is log-concave $($log-convex$)$. If for a certain $n\in\bN$ the real sequence 
$$
A_0, A_1, \ldots, A_{[n \slash 2]}
$$
has no more than one change of sign and $\sum_{0 \leq k \leq [n \slash 2]} A_k \ge0~(\le0)$, then 
\[
\sum_{0 \leq k \leq n \slash 2} f_k f_{n - k} A_k \ge0~(\le0). 
\] 
\end{lemma}

\begin{proof}
It is literally the same as that of~\cite[Lemma 2.1]{KK2013}. 
\end{proof}

The following is a refinement of \cite[Lemma~2]{KKJMAA2013}.
\begin{lemma} \label{LEM:extension}
Suppose a positive function $\mu\to{f(\mu)}$ satisfies the Tur\'{a}n type inequality    
$$
\Delta_{f}(1,1)=[f(\mu+1)]^2-f(\mu)f(\mu+2)\ge0 ~~(\le0)
$$
for all $\mu\ge0$.  Then
$$
\Delta_{f}(\alpha,\beta)=f(\mu+\alpha)f(\mu+\beta)-f(\mu)f(\mu+\alpha+\beta)\ge0~~(\le0)
$$
for all $\alpha,\beta\in\bN$ and $\mu\ge0$.
\end{lemma}

\begin{proof}
For the purposes of this proof it is convenient to write $\Delta_{f}(\alpha,\beta)=\Delta_{\mu}(\alpha,\beta)$ as $f$ will remain fixed while $\mu$ will vary.   We argue by induction. For $\alpha=\beta=1$ the inequality $\Delta_{\mu}(\alpha,\beta)\ge0$ holds by hypothesis of the lemma. Suppose it holds true for $\alpha,\beta\in\{1,2,\ldots,n\}$. Hence, taking $\alpha=n$, $\beta\in\{1,2,\ldots,n\}$ we have
$$
f(\mu+n)f(\mu+\beta)\ge f(\mu)f(\mu+n+\beta).
$$
By changing $\mu\to\mu+n$ and taking $\alpha=1$ we also obtain:
$$
f(\mu+n+1)f(\mu+n+\beta)\ge f(\mu+n)f(\mu+n+1+\beta).
$$   
Multiplying the above two inequalities we get
$$
f(\mu+n+1)f(\mu+\beta)\ge f(\mu)f(\mu+n+1+\beta).
$$
Hence, the required inequality holds for $\alpha=n+1$ and $\beta\in\{1,2,\ldots,n\}$. In view of the symmetry we also covered the case $\beta=n+1$ and $\alpha\in\{1,2,\ldots,n\}$. Then, multiplying the inequalities $\Delta_{\mu}(n,n+1)\ge0$ and $\Delta_{\mu+n}(1,n+1)\ge0$ we obtain $\Delta_{\mu}(n+1,n+1)\ge0$, so that we proved that
$\Delta_{\mu}(\alpha,\beta)\ge0$ for all $\alpha,\beta\in\{1,2,\ldots,n+1\}$. 
\end{proof}
\begin{remark}
An analogous lemma can also be formulated for $\alpha,\beta$ in any lattice not necessarily $\mathbb{N}^2$, but we will not need this fact here.     
\end{remark}

For the  coefficient-wise logarithmic concavity (convexity), however, we only managed to establish the following weaker claim.
\begin{lemma} \label{LEM:extensionCW}
Suppose a power series of the form \eqref{eq:f-defined} is coefficient-wise log-concave \emph{(}log-convex\emph{)} for $\mu\ge0$ and the  shifts $\alpha=1$ and $\beta=\{1,2,\ldots,n\}$, \ie, $\delta_k\ge0$ \emph{(}$\delta_k\le0$\emph{)} for all $k=0,1,\ldots$, where
$$
\Delta_{f}(\alpha,\beta;x)=f(\mu+\alpha;x)f(\mu+\beta;x)-f(\mu;x)f(\mu+\alpha+\beta;x)=\sum\limits_{k=0}^{\infty}\delta_k x^k.
$$
Then it is coefficient-wise log-concave \emph{(}log-convex\emph{)}  for shifts $\alpha,\beta\in\bN$ satisfying $\alpha+\beta\le n+1$.
\end{lemma}
\begin{proof}
Indeed, it is straightforward to check that for any $\beta\ge\alpha$ and $\delta\ge0$
\begin{equation}\label{eq:TuranianIdentity}
\Delta_{\mu}(\alpha+\delta,\beta;x)=\Delta_{\mu}(\alpha,\beta+\delta;x)+\Delta_{\mu+\alpha}(\delta,\beta-\alpha;x),
\end{equation}
where we again used the notation $\Delta_{\mu}(\alpha,\beta;x)=\Delta_{f}(\alpha,\beta;x)$.
Suppose the coefficients of $\Delta_{\mu}(1,\beta)$ are non-negative for~$\beta\in\{1,\ldots,n\}$. By taking $\alpha=\delta=1$ in \eqref{eq:TuranianIdentity} and letting $\beta$ run over  $\{1,\ldots,n-1\}$, we conclude that the power series coefficients of 
$\Delta_{\mu}(2,\beta)$ are non-negative for $\beta$ in this range.  Taking $\alpha=2$, $\delta=1$ in \eqref{eq:TuranianIdentity} and letting  $\beta$ run over  $\{2,\ldots,n-2\}$, we conclude that the power series of  
$\Delta_{\mu}(3,\beta)$ are non-negative for $\beta$ in this range.  Continuing in the same fashion, we cover the triangular set $\alpha+\beta\le n+1$ with $\beta\ge\alpha$.  The proof is completed by exchanging the roles of $\alpha$ and $\beta$ in the above argument.
\end{proof}

It remains open whether the conclusion in the above lemma can be extended to all natural shifts $\alpha,\beta\in\bN$ in analogy with Lemma~\ref{LEM:extension}.  Lemma~\ref{LEM:extensionCW} also leads to  a strengthening of \cite[Lemma~3]{KKJMAA2013} and various results based on it. The result reads as follows.

\begin{corollary}
Suppose a power series of the form \eqref{eq:f-defined} is coefficient-wise log-concave \emph{(}log-convex\emph{)} for $\mu\ge0$ and the  shifts $\alpha=1$ and $\beta\ge1$.  Then it is coefficient-wise log-concave \emph{(}log-convex\emph{)} for $\mu\ge0$ and all $\alpha,\beta\ge1$ such that at least one of them is an integer.  
\end{corollary}
\begin{proof}
We proceed by induction. The claim holds for $\alpha=1$, $\beta\ge1$ by assumption. Suppose it holds for $\alpha=1,2,\ldots,k$, $\beta\ge1$, then
 taking $\alpha=1$, $\beta=k+1$ and any $\delta>0$ in \eqref{eq:TuranianIdentity} we have
$$
\Delta_{\mu}(1+\delta,k+1;x)=\Delta_{\mu}(k+1,1+\delta;x)=\Delta_{\mu}(1,k+1+\delta;x)+\Delta_{\mu+1}(k,\delta;x)
$$
which establishes the claim for $\alpha=k+1$, $\beta\ge1$. The symmetry 
$$\Delta_{\mu}(\alpha,\beta)=\Delta_{\mu}(\beta,\alpha)$$
 completes the proof.
\end{proof}

\section{Multiple factorial series} \label{SEC:mps}

Define a generalization of \eqref{eq:KS2010}  as follows:
\begin{equation}\label{eq:multiplePoch-g}
g(\mu; x)= \sum_{n=0}^{\infty}g_n \frac{(\mu)_{nr}}{(nr)!}x^n,~~~r=2,3,\ldots.    
\end{equation}
Note that 
\begin{equation}\label{eq:logDmultiplePoch}
[\log((\mu)_{nr})]''=\psi'(\mu+rn)-\psi'(\mu)\le0,    
\end{equation}
where $\psi(z)=\Gamma'(z)/\Gamma(z)$ denotes the digamma function \cite[p.13]{AAR}. 
The equality  above is only attained for $n=0$ in view of the inequalities
$$
\psi'(x)=\int_{0}^{\infty}\frac{te^{-tx}}{1-e^{-t}}dt>0,~~~~\psi''(x)=-\int_{0}^{\infty}\frac{t^2e^{-tx}}{1-e^{-t}}dt<0.
$$
Hence, we are dealing the the infinite sum of log-concave functions which may be log-concave, log-convex or neither. 
Setting $g_n=C>0$ here does not lead to a log-neutral function of $\mu$ (as will be explicitly seen below), so that there is no hope to get a duality between log-concave/log-convex sequences $\{g_n\}$ and log-concave/log-convex functions as outlined in the Introduction. Ahn Ninh (private communication, 2017) suggested to modify \eqref{eq:multiplePoch-g}  as follows 
\begin{equation}\label{eq:multiplePoch-f}
\mu\to f(\mu; x) = \sum_{n=1}^{\infty} f_n \frac{(\mu)_{nr}}{(nr-1)!}x^n
\end{equation}
and conjectured that this function is log-concave if $\{f_n\}_{n\ge0}$ is log-concave. See related developments in \cite{NinhPham2018}. In this section we will prove this conjecture for $r=2$ and disprove numerically for $r=4$ (strongly suggesting that is it also wrong for $r>4$). The case $r=3$ remains open and is formulated in the form of Conjecture~\ref{CONJ:1}.  We will further proof that $\mu\to g(\mu; x)$ defined in \eqref{eq:multiplePoch-g} is coefficient-wise log-convex for $r=2$ and is neither log-convex nor log-concave for $r=3$ (strongly suggesting that it remains so for $r>3$).
We will need the following.

\begin{lemma} \label{LEM:onesignchange}
Let $k, m\in \bN_0$ with $k \leq m/2$, $\mu \geq 0$ and $\alpha,\beta>0$. Set 
\begin{multline} \label{EQ:onesignchange}
T_{k, m} := (\mu + \alpha)_k (\mu + \beta)_{m-k} + (\mu + \alpha)_{m-k} (\mu + \beta)_{k} \\
 -  (\mu)_k (\mu + \alpha + \beta)_{m-k} -  (\mu)_{m-k} (\mu +\alpha + \beta)_{k}. 
\end{multline} 
If $T_{k,m} \leq 0$ for some $1\le k\le m/2$,  then $T_{k-1, m} < 0$.
\end{lemma}
\begin{proof}
See~\cite[proof of Theorem~1]{KSJMAA2010}.
\end{proof}

\begin{theorem} \label{THM:logconvexitysecondhyper}
Let 
\begin{equation}\label{eq:g-defined}
g(\mu; x) = \sum_{n = 0}^\infty g_n \frac{(\mu)_{2 n}}{(2 n)!} x^n=\sum_{n = 0}^\infty g_n \frac{(\mu/2)_{n}((\mu+1)/2)_{n}}{(1/2)_nn!} x^n,
\end{equation}
where $\{g_n\}_{n = 0}^\infty$ is a log-convex sequence independent of $\mu$. Then the formal power series $g(\mu; x)$ is coefficient-wise log-convex for $\mu>0$. 
\end{theorem}
\begin{proof}
For each $\alpha,\beta>0$ we have 
\[
g(\mu + \alpha; x) g(\mu + \beta; x) - g(\mu; x) g(\mu + \alpha + \beta; x) := \sum_{m = 1}^\infty \phi_m x^m,
\]
where $\phi_m =  \sum_{k = 0}^m g_k g_{m - k} M_k$ and 
\[
M_k = \frac{1}{(2 k)! \left(2 (m - k)\right)!} \left[ (\mu + \alpha)_{2 k} (\mu + \beta)_{2 (m - k)} - (\mu)_{2 k} (\mu + \alpha + \beta)_{2 (m - k)} \right].
\]
Using Gauss pairing, we may write $\phi_m$ in the following form: 
\[
\phi_m = \sum_{k = 0}^{[m/2]} g_k g_{m - k} A_k,
\]
where $A_k = M_k + M_{m - k}$ for $k < m/2$, and $A_k = M_k$ for $k = m/2$. Set 
\begin{equation} \label{EQ:tildeA}
\tilde{A}_k = (2 k)! (2 (m - k))! A_k,
\end{equation}
which has the same sign as $A_k$. 
Then 
\[ 
 \tilde{A}_k =
\begin{cases} 
      \underbrace{(\mu + \alpha)_{2 k} (\mu + \beta)_{2 (m - k)}}_{u_k} + \underbrace{(\mu + \alpha)_{2 (m - k)} (\mu + \beta)_{2 k}}_{v_k} \\
      -\underbrace{(\mu)_{2 k} (\mu + \alpha + \beta)_{2 (m - k)}}_{r_k} - \underbrace{(\mu)_{2 (m - k)} (\mu + \alpha + \beta)_{2 k}}_{s_k} \quad \text{ if } \quad k < m/2, \\
      (\mu + \alpha)_{2 k} (\mu + \beta)_{2 k} - (\mu)_{2 k} (\mu + \alpha + \beta)_{2 k} \quad \quad \quad \quad \quad \quad \text{ if } \quad k = m/2. 
   \end{cases}
\]
First, we show that $\sum_{k = 0}^{[m/2]} A_k < 0$ for $m \geq 1$. To this end, we set 
\begin{equation} \label{EQ:logconcavitypsi}
\psi(\mu; x) := \sum_{n = 0}^\infty \frac{(\mu)_{2 n}}{(2 n)!} x^n=\frac{1}{2} \left[\frac{1}{(1 + \sqrt{x})^\mu} + \frac{1}{(1 - \sqrt{x})^\mu} \right],
\end{equation}
where the second equality is immediate by expanding each of the two summands by the binomial theorem.
Then
\begin{align*}
\xi_\mu(\alpha, \beta; x) & :=  \psi(\mu + \alpha; x) \psi(\mu + \beta; x) - \psi(\mu; x) \psi(\mu + \alpha + \beta; x) \\
&= \sum_{m = 1}^{\infty}x^m\sum_{k = 0}^{m} M_k = \sum_{m = 1}^{\infty}x^m \sum_{k = 0}^{[m/2]} A_k. 
\end{align*}
Using~\eqref{EQ:logconcavitypsi}, we find that 
\begin{equation} \label{EQ:logconcavitypsi3}
\xi_\mu(\alpha, \beta; x) = -\frac{1}{4} (1- x)^{-\mu} \left[(1 - \sqrt{x})^{-\alpha}-(1 + \sqrt{x})^{-\alpha}\right] \left[(1 - \sqrt{x})^{-\beta} -(1 + \sqrt{x})^{-\beta}\right]. 
\end{equation}
Since 
\begin{equation*}
(1+\sqrt{x})^{-\alpha} - (1-\sqrt{x})^{-\alpha} 
=\sum_{n = 0}^\infty {-\alpha \choose n} (\sqrt{x})^n - \sum_{n = 0}^\infty {-\alpha \choose n} (- \sqrt{x})^n 
= 2 \sqrt{x} \sum_{n = 0}^\infty {-\alpha \choose 2 n + 1} x^n
\end{equation*}
has only negative coefficients, it follows from~\eqref{EQ:logconcavitypsi3} that $\psi(\mu; x)$ is coefficient-wise log-convex, 
and thus $\sum_{k = 0}^{[m/2]} A_k < 0$ for $m \geq 1$.

Next, we prove that for $m \geq 2$ the sequence $\tilde{A}_0, \ldots, \tilde{A}_{[m/2]}$ has no more than one change of sign. 
Assume that $\tilde{A}_{k} \leq 0$ for some $k < m/2$.  Then $\tilde{A}_{k} = T_{2k, 2m}$, which is given in~\eqref{EQ:onesignchange}. 
Thus, it follows from Lemma~\ref{LEM:onesignchange} that
$$\tilde{A}_{k - 1} = T_{2k - 2, 2m} < 0.$$

Hence, since $\sum_{k = 0}^{[m/2]} A_k<0$, and the sequence $A_0, \ldots, A_{[m/2]}$ 
has no more than one change of sign, we conclude by Lemma~\ref{LEM:keylemma} that $g(\mu; x)$ is coefficient-wise log-convex. 
\end{proof}

\begin{remark}
Note that the function $\psi(\mu;x)$ defined in \eqref{EQ:logconcavitypsi} is a special case of the so-called hypergeometric superhyperbolic cosine~\cite[page 74, Definition 2.10]{Yang2021}. Its second logarithmic derivative is easily seen to be positive:
$$
\frac{\partial^2}{\partial\mu^2}(\log\psi(\mu;x))=\frac{4 (1 - x)^{\mu}(\mathrm{arctanh}\sqrt{x})^2}{((1 - \sqrt{x})^{\mu} + (1 + \sqrt{x})^{\mu})^2}>0,
$$
which proves its log-convexity.  On the other hand, the function $\mu\to(\mu)_{2n}$ in \eqref{eq:g-defined} is  log-concave for every $n\in\mathbb{N}_0$ according to \eqref{eq:logDmultiplePoch}. This shows that the condition that $\{g_n\}$ is log-convex cannot be dropped, since taking $g_k=1$ with $g_j=0$ for $j\ne{k}$ yields a log-concave function.   Moreover, the sequence $A_0,A_1,\ldots,A_{[m/2]}$ has exactly one change of sign, which shows that restricting all $g_n$ to be strictly positive (which rules out the previous example) is also insufficient, emphasizing the importance of the log-convexity of $\{g_n\}$.  A proof of exactly one change of sign is given in the Appendix.
\end{remark}

\begin{remark}
Theorem~\ref{THM:logconvexitysecondhyper} is a natural generalization of~\cite[Theorem 1]{KSJMAA2010}. Rewriting the proofs  of~\cite[Theorems 2 and 3]{KSJMAA2010} \emph{mutatis mutandis}, we can similarly show that for each integer $r \geq 1$ the formal power series 
\begin{align*}
\mu\to\sum_{n = 0}^\infty f_n \Gamma(\mu + r n) x^n, \\
\mu\to\sum_{n = 0}^\infty \frac{f_n}{(\mu)_{r n}} x^n,
\end{align*}
are coefficient-wise log-convex for arbitrary non-negative sequence
 $\{f_n\}_{n = 0}^\infty$ independent of $\mu$ \emph{(}here $\Gamma(\cdot)$ stands for 
Euler's gamma function\emph{)}. 
\end{remark}

\begin{remark} \label{REM:r3}
Numerical experiments suggest that $\mu \to \sum_{n=0}^{\infty} \frac{(\mu)_{nr}}{(nr)!}x^n$ is neither coefficient-wise log-convex nor coefficient-wise log-concave for $r \geq 3$. Consider, for instance, 
\[
\psi_3(\mu;x)=\sum_{n=0}^{\infty}\frac{(\mu)_{3n}}{(3n)!}x^n
=\frac{1}{3}\sum_{k=0}^{2}\frac{1}{(1-\omega_kx^{1/3})^{\mu}},
\]
where $\omega_k=e^{2k\pi{i}/3}$ are the primitive roots of unity and the second equality follows from the binomial theorem combined with the elementary identity
$$
\sum\limits_{k=0}^{r-1}\omega_k^n=
\begin{cases}
r, & \text{if } r \text{ divides } n
\\
0, & \text{otherwise}
\end{cases}
$$
for $\omega_k=e^{2k\pi{i}/r}$, $k=0,\ldots,r-1$, being primitive $r$-th roots of unity. Note first, that the function $\mu\to\psi_3(\mu;x)$ is convex for $x>0$ as a sum of convex functions. Next, by direct calculation  we have
\begin{multline*}
\psi_3(\mu + \alpha; x) \psi_3(\mu + \beta; x) - \psi_3(\mu; x)\psi_3(\mu+\alpha+\beta; x)
\\
= -\frac{1}{2} \alpha  \beta  (\alpha +\beta +2 \mu +2) x + c(\mu, \alpha, \beta) x^2 + O(x^3)
\end{multline*}
 as $x\to0$,  where 
\begin{multline*}
 c(\mu, \alpha, \beta) = -\frac{1}{240}\alpha\beta\left(2\alpha^4+5\alpha^3\beta +10\alpha^3\mu +25\alpha^3+30\alpha^2\beta
 +60\alpha^2\mu +100\alpha^2 \right. \\
 \left. +5\alpha\beta^3+30\alpha\beta^2+30\alpha\beta\mu +110\alpha\beta  +30\alpha\mu^2 + 220\alpha\mu +185\alpha\right. \\
   \left.+2\beta^4+10\beta^3\mu +25\beta^3+60 \beta^2\mu +100\beta^2+30 \beta\mu^2 + 220\beta\mu +185\beta\right. \\
   \left. +20\mu^3+220\mu^2+370\mu +156 -30\alpha\beta\mu^2 -20\mu^3( \alpha +\beta) -10\mu^4\right). 
\end{multline*}
It is straightforward to verify that 
\begin{align*}
c(5.76788, 1, 1) & = -0.0998924, \\
c(6.13463, 1, 1) & = 9.99983.
\end{align*}
Thus, we see that $\psi_3(\mu; x)$ is neither coefficient-wise log-convex nor coefficient-wise log-concave.
\end{remark}

\begin{theorem} \label{THM:logconcavitysecondhyper}
Let 
\[
f(\mu; x) = \sum_{n=1}^\infty f_n \frac{(\mu)_{2n}}{(2n-1)!} x^n=x(\mu)_{2}\sum_{n=0}^\infty f_{n+1} \frac{(\mu/2+1)_{n}(\mu/2+3/2)_{n}}{(3/2)_nn!} x^n,
\]
where $\{ f_n \}_{n=1}^\infty$ is log-concave and independent of $\mu$. Then the formal power series $f(\mu; x)$ is coefficient-wise log-concave. 
\end{theorem}
\begin{proof} 
For each $\alpha,\beta>0$ we get
\[
f(\mu + \alpha; x) f(\mu + \beta; x) - f(\mu; x) f(\mu + \alpha + \beta; x)  := \sum_{m = 2}^\infty \tilde{\phi}_m x^m,
\]
where $\tilde{\phi}_m =  \sum_{k = 1}^m f_k f_{m - k} \tilde{M}_k$ and 
\[
\tilde{M}_k = \frac{(\mu+\alpha)_{2k}(\mu+\beta)_{2(m-k)} - (\mu)_{2k}(\mu+\alpha+\beta)_{2(m-k)} }{(2k-1)!\left(2(m-k)-1\right)!}.
\]
Moreover, we can write $\tilde{\phi}_m$ in the following form: 
\[
\tilde{\phi}_m = \sum_{k = 1}^{[m/2]} f_k f_{m - k} B_k,
\]
where $B_k = \tilde{M}_k + \tilde{M}_{m - k}$ for $k < m/2$, and $B_k = \tilde{M}_k$ for $k = m/2$. The  modified numbers
\[
\tilde{B}_k = (2k - 1)! \left(2 (m - k) - 1 \right)! B_k,
\]
have the same sign as $B_k$ and are obviously equal to $ \tilde{A}_k$ defined in~\eqref{EQ:tildeA}, $\tilde{B}_k = \tilde{A}_k$. 
Therefore, by the proof of Theorem~\ref{THM:logconvexitysecondhyper} and Lemma~\ref{LEM:onesignchange} we see that  the sequence $B_1, B_2, \ldots, B_{[m/2]}$ has no more than one 
change of sign and $B_{[m/2]} > 0$. 

Next, we prove that $\sum_{k = 1}^{[m/2]} B_k > 0$ for $m \geq 2$. Set 
\begin{equation*}
\tilde{\psi}(\mu; x):= \sum_{n = 1}^\infty \frac{(\mu)_{2 n}}{(2 n - 1)!} x^n,~~~
\eta(\mu; x):=\frac{1}{(1-\sqrt{x})^{\mu+1}}-\frac{1}{(1+\sqrt{x})^{\mu+1}}. 
\end{equation*}
Using the binomial theorem it is easy to see that  
$$
\tilde{\psi}(\mu; x) = \frac{\mu\sqrt{x}}{2}\eta(\mu; x).
$$ 
Thus, we have 
\begin{multline}
\tilde{\psi}(\mu + \alpha; x) \tilde{\psi}(\mu + \beta; x) - \tilde{\psi}(\mu; x) \tilde{\psi}(\mu + \alpha + \beta; x) = \sum_{m = 2}^\infty \left(\sum_{k = 0}^{[m/2]} B_k\right) x^m \\
= \frac{(\mu + \alpha) (\mu + \beta) x}{4} \eta(\mu + \alpha; x) \eta(\mu + \beta; x) - \frac{\mu (\mu + \alpha + \beta) x}{4} \eta(\mu; x) \eta(\mu + \alpha + \beta; x).  \label{EQ:coeff3}
\end{multline}
Since $(\mu + \alpha) (\mu + \beta) > \mu (\mu + \alpha + \beta)$, we see that the coefficients of~\eqref{EQ:coeff3} are larger than 
those of 
\[
\frac{\mu (\mu + \alpha + \beta) x}{4}\left[ \eta(\mu + \alpha; x) \eta(\mu + \beta; x) -  \eta(\mu; x) \eta(\mu + \alpha + \beta; x)\right]. 
\]
Therefore, it suffices to prove the coefficient-wise  log-concavity of $\eta(\mu; x)$. Indeed, we have 
\begin{multline} \label{EQ:coeff4}
\eta(\mu + \alpha; x) \eta(\mu + \beta; x) -  \eta(\mu; x) \eta(\mu + \alpha + \beta; x) \\
= (1 - x)^{-1 - \mu} \left[(1 +\sqrt{x})^{-\alpha}  - (1 - \sqrt{x})^{-\alpha} \right]  \cdot \left[(1 + \sqrt{x})^{-\beta}  - (1 - \sqrt{x})^{-\beta} \right].
\end{multline}
Since the coefficients of $\left[(1 +\sqrt{x})^{-\alpha}  - (1 - \sqrt{x})^{-\alpha}\right]$ are negative, as explained in the proof of Theorem~\ref{THM:logconvexitysecondhyper}, it follows from~\eqref{EQ:coeff4} 
that $\eta(\mu; x)$ is coefficient-wise log-concave, and, therefore, $\sum_{k = 0}^{[m/2]} B_k > 0$ for $m \geq 2$ by~\eqref{EQ:coeff3}.

Summarizing, since $\sum_{k = 1}^{[m/2]} B_k > 0$, the sequence $B_1, B_2, \ldots, B_{[m/2]}$ has no more than one 
change of sign and $B_{[m/2]} > 0$, it follows from Lemma~\ref{LEM:keylemma} that $f(\mu; x)$ is coefficient-wise log-concave. 
\end{proof}

\begin{remark}
Similarly to Remark~\ref{REM:r3}, we find that the function
$$
\mu \to \tilde{\psi}_r(\mu; x) := \sum_{n=1}^{\infty} \frac{(\mu)_{nr}}{(nr-1)!}x^n=\frac{\mu}{r}x^{1/r}\sum_{k=0}^{r-1}\frac{\omega_k}{(1-\omega_kx^{1/r})^{\mu+1}},
$$ 
$\omega_k=e^{2\pi{i}k/r}$, is neither coefficient-wise log-convex nor coefficient-wise log-concave for $r \geq 4$ by numerical computations. However, the numerical evidence confirms the following conjecture.  
\end{remark}
\begin{conjecture} \label{CONJ:1}
 Assume that $\{f_n\}_{n = 0}^\infty$ is log-concave and independent of $\mu$. 
Then the function 
\[
\mu\to f(\mu; x) = \sum_{n=1}^\infty f_n \frac{(\mu)_{3n}}{(3n - 1)!}x^n
\]
is coefficient-wise log-concave.
\end{conjecture}

\section{The series in the functions $\phi_k(\mu)=(\mu)_{2k}/(\mu+1)_k$} \label{SEC:sps}
We consider the series
\begin{equation} \label{EQ:firstconcavity}
h(\mu; x) := \sum_{k = 0}^\infty h_k \frac{(\mu)_{2 k}}{(\mu + 1)_k k!} x^k
=\sum_{k = 0}^\infty h_k \frac{(\mu/2)_{k}((\mu+1)/2)_{k}}{(\mu + 1)_k k!} (4x)^k.
\end{equation}
Along with the series~\eqref{eq:KS2010}, it gives another example of log-neutral function when $h_k=1$ for all $k$.
Indeed, if we set  
$$
\lambda(\mu;x):=\sum_{k=0}^\infty \frac{(\mu)_{2 k}}{(\mu + 1)_k k!} x^k \quad (|4 x| < 1),
$$
then $\lambda(\mu; x)$ is a hypergeometric function with the following properties. 
\begin{lemma}~\cite[Lemma 7]{Shibukawa2019}\label{LEM:firsthypergeometric}
The hypergeometric function $\lambda(\mu; x)$ has the following properties. 
\begin{enumerate}
\item Closed form
\[
\lambda(\mu; x) = \left(\frac{1 - \sqrt{1 - 4 x}}{2 x}\right)^{\mu}.
\]
\item Index law
\[
\lambda(\alpha; x) \lambda(\beta; x) = \lambda(\alpha + \beta;x).
\]
\end{enumerate}
\end{lemma}
This makes the following conjecture very natural.

\begin{conjecture} \label{CONJ:2}
 Assume that $\{h_k\}_{k = 0}^\infty$ is log-concave \emph{(}log-convex\emph{)} and independent of $\mu$. 
Then the function $\mu\to h(\mu;x)$ defined in \eqref{EQ:firstconcavity} is coefficient-wise log-concave \emph{(}log-convex\emph{)}.
\end{conjecture}

Forming the generalized Tur\'{a}nian
\begin{equation*}
\Delta_{h}(\alpha,\beta;x)=h(\mu+\alpha;x) h(\mu+\beta;x) - h(\mu; x)h(\mu+\alpha+\beta; x)=\sum_{m = 2}^\infty \delta_m x^m,
\end{equation*}
we will have 
\begin{align*}
\delta_m&=\sum_{k = 0}^m h_k h_{m - k} M_k, 
\\
M_k&=\frac{1}{k! (m - k)!} \left[\frac{(\mu + \alpha)_{2k} (\mu + \beta)_{2 (m - k)}}{(\mu + \alpha + 1)_k (\mu + \beta + 1)_{m - k}} - \frac{(\mu)_{2k} (\mu + \alpha + \beta)_{2 (m - k)}}{(\mu +  1)_k (\mu +\alpha +  \beta + 1)_{m - k}} \right].
\end{align*}
Furthermore, we can write 
\begin{equation*}
\delta_m=\sum_{k = 0}^{[m/2]}h_k h_{m - k} A_k \ \ \text{ for each } \ \ m \geq 2,
\end{equation*}
where $A_k=M_k + M_{m-k}$ for $k < m \slash 2$ and $A_k=M_k$ for $k=m\slash 2$. 
As
\begin{equation*}
\sum_{m = 0}^\infty \left(\sum_{k = 0}^{[m/2]} A_k\right) x^m = \sum_{m = 0}^\infty \left(\sum_{k = 0}^{m} M_k \right) x^m 
= \lambda(\mu + \alpha; x) \lambda(\mu + \beta; x) - \lambda(\mu; x) \lambda(\mu + \alpha + \beta; x),
\end{equation*}
Lemma~\ref{LEM:firsthypergeometric}  implies that $\sum_{k = 0}^{[m/2]} A_k = 0$ for $m \geq 0$. Hence, to apply Lemma~\ref{LEM:keylemma} we only need to prove that the sequence $\{A_k\}_{k=0}^{[m/2]}$ has no more than one change of sign.  So far, we failed to prove this fact in general.  Nevertheless, we established the following. 

\begin{theorem} \label{THM:discretelogconcavity}
Let $h(\mu; x)$ be the formal power series defined by~\eqref{EQ:firstconcavity}. If $\{ h_k \}_{k = 0}^\infty$ is log-concave \emph{(}log-convex\emph{)} and independent of $\mu$, then $\mu \to h(\mu;x)$ is log-concave \emph{(}log-convex\emph{)} for integer shifts, i.e. $\Delta_h(\alpha,\beta;x)\ge0$ \emph{(}$\Delta_h(\alpha,\beta;x)\le0$\emph{)} for $\alpha,\beta\in\bN$ and coefficient-wise  log-concave \emph{(}log-convex\emph{)} if furthermore $\alpha+\beta\le4$.
\end{theorem}

\begin{proof}
Set $\alpha=\beta=1$ and consider the following generalized Tur\'{a}nian
\begin{equation*}
\Delta_{h}(1, 1; x)= h(\mu +1;x)^2 - h(\mu;x) h(\mu +2;x)=\sum_{m = 2}^\infty \delta_m x^m,
\end{equation*}
where
\begin{align*}
\delta_m & = \sum_{k = 0}^m h_k h_{m - k} M_k, \\
M_k    & =   \frac{1}{k! (m - k)!} \left[\frac{(\mu + 1)_{2k} (\mu + 1)_{2 (m - k)}}{(\mu + 2)_k (\mu + 2)_{m - k}} - \frac{(\mu)_{2k} (\mu + 2)_{2 (m - k)}}{(\mu +  1)_k (\mu + 3)_{m - k}} \right].
\end{align*}
Gauss pairing yields: 
\begin{align*}
\delta_m & = \sum_{k = 0}^{[m/2]} h_k h_{m - k} A_k \ \ \text{ for each } \ \ m \geq 2,
\end{align*}
where $A_k = M_k + M_{m - k}$ for $k < m \slash 2$ and $A_k = M_k$ for $k = m \slash 2$. 
For each $k \geq 0$, we set $\tilde{A}_k = k! (m - k)! A_k$, which has the same sign as $A_k$. Then 
\[ 
 \tilde{A}_k =
\begin{cases} 
      2 \cdot \underbrace{\frac{(\mu + 1)_{2 k} (\mu + 1)_{2 (m - k)}}{(\mu + 2)_k (\mu + 2)_{m - k}}}_{u_k}   -\underbrace{\frac{(\mu)_{2 k} (\mu + 2)_{2 (m - k)}}{(\mu + 1)_{k} (\mu +3)_{m - k}}}_{r_k} 
      \\
      -   \underbrace{\frac{(\mu )_{2  (m - k)} (\mu + 2)_{2 k}}{(\mu + 1)_{m - k}  (\mu + 3)_{k}}}_{s_k},  \quad \quad \text{ if }  \quad k < m/2, \\
      \dfrac{(\mu + 1)_{2 k}^2}{(\mu + 2)_k^2}  -  \dfrac{(\mu )_{2  k} (\mu + 2)_{2 k}}{(\mu + 1)_{k}  (\mu + 3)_{k}}, \quad \quad \text{ if } \quad k = m/2.
   \end{cases}
\]
Our next goal is to show that the sequence $\{\tilde{A}_k\}_{k = 0}^{[m/2]}$  has no more than one change of sign, \ie, 
\begin{equation} \label{EQ:onechangeofsign}
\tilde{A}_k \leq 0 \quad \Rightarrow \quad  \tilde{A}_{k - 1} < 0 \ \ \text{ for } \ \  k \geq 1. 
\end{equation}
Since for $m\ge2$
\begin{align*}
\tilde{A}_0 & = 2 \frac{(\mu + 1)_{2 m}}{(\mu + 2)_m} - \frac{(\mu + 2)_{2 m}}{(\mu + 3)_m} - \frac{(\mu)_{2 m}}{(\mu + 1)_m} \\
& = \frac{(\mu + 2)_{2 m - 2}}{(\mu + 3)_{m - 2}} \left[ 2 \frac{(\mu + 1) (\mu + 2 m)}{(\mu + 2) (\mu + m + 1)} 
- \frac{(\mu + 2 m) (\mu + 2 m + 1)}{(\mu + m +1) (\mu + m + 2)} - \frac{\mu}{\mu + 2} \right] \\ 
& = - \frac{(\mu + 2)_{2 m - 2}}{(\mu + 3)_{m - 2}} \cdot \frac{(m - 1) m ( \mu + 4)}{(\mu + 2) (\mu + m +1) (\mu + m + 2)} < 0,
\end{align*}
we see that the claim~\eqref{EQ:onechangeofsign} holds for $k = 1$. If $k \geq 2$, then we have 
\begin{equation*}
\tilde{A}_k= 2 u_k - r_k - s_k = \frac{(\mu+2)_{2 k - 2} (\mu + 2)_{2 (m - k) - 2}}{(\mu + 3)_{k - 2} (\mu + 3)_{(m - k) - 2}} \left(2 g_1 - g_2 - g_3 \right),
\end{equation*}
where 
\begin{align*}
g_1 & = \frac{(\mu + 1)^2 (\mu + 2 k) (\mu + 2 (m -k))}{(\mu + 2)^2 (\mu + k + 1) (\mu + (m - k) +1)}, \\
g_2 & = \frac{\mu (\mu + 2 (m - k)) (\mu + 2 (m - k) +1)}{(\mu + 2) (\mu + (m - k) + 1) (\mu + (m - k) + 2)}, \\
g_3 & = \frac{\mu (\mu + 2 k) (\mu + 2 k + 1)}{(\mu + 2) (\mu + k + 1) (\mu + k + 2)}.
\end{align*}
Set $q=k-2$ and $t=m-2k$. Then $q \geq 0$ and $t \geq 0$. A direct calculation implies that 
\[
2 g_1 - g_2 - g_3 = \frac{n_1}{(\mu + 2)^2 (\mu + q + 3) (\mu + q + 4) (\mu + q + t + 3) (\mu + q + t + 4)},
\]
where 
\begin{multline} \label{EQ:n1}
n_1 = 4 \mu ^4+52 \mu ^3+256 \mu ^2+576 \mu +8 q^4+4 \mu ^2 q^3+32 \mu  q^3+16 q^3 t+96 q^3+6 \mu ^3 q^2 \\
+66 \mu ^2 q^2+264 \mu  q^2+8 q^2 t^2+6 \mu ^2 q^2 t+48 \mu  q^2 t+144 q^2
   t+416 q^2+2 \mu ^4 q+38 \mu ^3 q \\ +244 \mu ^2 q+688 \mu  q+2 \mu ^2 q t^2+16 \mu  q t^2+48 q t^2+6 \mu ^3 q t+66 \mu ^2 q t+264 \mu  q t+416 q t+768 q \\ -\mu ^4 t^2-7 \mu ^3
   t^2-10 \mu ^2 t^2+24 \mu  t^2+64 t^2+\mu ^4 t+19 \mu ^3 t+122 \mu ^2 t+344 \mu  t+384 t+512.
\end{multline}
Thus, we see that $\tilde{A}_k \leq 0$ is equivalent to $n_1 \leq 0$ with $q \geq 0, t \geq 0$, and  $\mu \geq 0$. 
Set
\begin{align*}
I_1 & = \frac{u_{k - 1}}{u_k} = \frac{(k+\mu +1) (2 k-\mu -2 m-2) (2 k-\mu -2 m-1)}{(2 k+\mu -1) (2 k+\mu ) (-k+\mu +m+2)}, \\
I_2 & = \frac{r_{k - 1}}{r_k} = \frac{(k+\mu ) (2 k-\mu -2 m-3) (2 k-\mu -2 m-2)}{(2 k+\mu -2) (2 k+\mu -1) (-k+\mu +m+3)}, \\
I_3 & = \frac{s_{k - 1}}{s_k} = \frac{(k+\mu +2) (2 k-\mu -2 m-1) (2 k-\mu -2 m)}{(2 k+\mu ) (2 k+\mu +1) (-k+\mu +m+1)}. 
\end{align*}
Then 
\begin{align*}
\tilde{A}_{k - 1} & = 2 u_{k - 1} - r_{k - 1} - s_{k - 1} = 2 I_1 u_k - I_2 r_k - I_3 s_k \\
& =  \frac{(\mu+2)_{2 k - 2} (\mu + 2)_{2 (m - k) - 2}}{(\mu + 3)_{k - 2} (\mu + 3)_{(m - k) - 2}} \left(2 I_1 g_1 -  I_2 g_2 - I_3 g_3 \right),
\end{align*}
where 
\begin{multline*}
2 I_1 g_1 -  I_2 g_2 - I_3 g_3 = \frac{(\mu + 2 q + 2 t + 4) (\mu + 2 q + 2 t + 5) n_2}{(\mu + 2)^2 (\mu + q + 3) (\mu + 2 q + 2) (\mu + 2 q + 3)} \\
\times \frac{1}{(\mu + q + t + 3) (\mu + q + t + 4) (\mu + q + t + 5)}
\end{multline*}
with
\begin{multline} \label{EQ:n2}
n_2 = 18 \mu ^3+150 \mu ^2+408 \mu +8 q^4+4 \mu ^2 q^3+32 \mu  q^3+16 q^3 t+96 q^3+6 \mu ^3 q^2+66 \mu ^2 q^2 \\+264 \mu  q^2+8 q^2 t^2+6 \mu ^2 q^2 t+48 \mu  q^2 t+128 q^2 t+400 q^2+2
   \mu ^4 q+38 \mu ^3 q+240 \mu ^2 q+656 \mu  q \\ +2 \mu ^2 q t^2+16 \mu  q t^2+32 q t^2+6 \mu ^3 q t+62 \mu ^2 q t+232 \mu  q t+304 q t+672 q-\mu ^4 t^2-7 \mu ^3 t^2 \\ -12 \mu ^2
   t^2+8 \mu  t^2+24 t^2-3 \mu ^4 t-15 \mu ^3 t+14 \mu ^2 t+160 \mu  t+192 t+360. 
\end{multline}
Therefore, we conclude that $\tilde{A}_{k - 1} < 0$ is equivalent to $n_2 < 0$ with $q \geq 0, t \geq 0$, and  $\mu \geq 0$. 
By~\eqref{EQ:n1} and~\eqref{EQ:n2}, we have
\begin{multline} \label{EQ:niceid}
n_2 = n_1 -4\mu^4-34\mu^3 - 106\mu^2 - 168\mu -16q^2t - 16q^2 -4\mu^2q - 32{\mu}q 
\\
- 16qt^2 - 4\mu^2qt -32{\mu}qt - 112qt- 96q - 2\mu^2t^2 - 16{\mu}t^2 - 40t^2 - 4\mu^4t- 34\mu^3t 
\\
- 108\mu^2t- 184\mu{t} - 192t  -152.
\end{multline}
From the above identity, we see that $n_1 \leq 0$ implies $n_2 < 0$, and thus the claim~\eqref{EQ:onechangeofsign} holds for $k \geq 2$. 

By Lemma~\ref{LEM:firsthypergeometric}, we see that $\sum_{k = 0}^{[m/2]} A_k =0$. Combined with the fact that the sequence $\{A_k\}_{k=0}^{[m/2]}$ has no more than one change of sign, we conclude that $A_{[m/2]} > 0$. 
Then, it follows from Lemma~\ref{LEM:keylemma} that the claim of the theorem holds for $\alpha=\beta=1$.  According to Lemma~\ref{LEM:extension}, we established discrete logarithmic concavity (convexity for log-convex $\{h_k\}$) for all shifts $\alpha,\beta\in\bN$.

Similar approach can be applied to verify that the sequence  $\{\tilde{A}_k\}_{k = 0}^{[m/2]}$  has no more than one change of sign for $(\alpha, \beta)\in\{(1,2), (1,3)\}$. The problem still amounts to proving that  $n_1 \leq 0$ implies $n_2 < 0$ with  $q \geq 0, t \geq 0$, and  $\mu \geq 0$, where $n_1$ and $n_2$ are certain polynomials in $q, t$, 
and~$\mu$. However, in those two cases, the relation connecting $n_2$ with $n_1$ has terms of opposite signs, unlike~\eqref{EQ:niceid}. Instead, we utilized Cylindrical Algebraic Decomposition~\cite{Collins1975, Kauers2011} as implemented in {\tt Mathematica} to prove the desired inequalities. Detailed calculations can be found in the supplemental material~\cite{YZ2023}.
Finally, coefficient-wise logarithmic concavity (convexity for log-convex $\{h_k\}$) for shifts $\alpha+\beta\le4$ follows from Lemma~\ref{LEM:extensionCW}.
\end{proof}

\section{The series in the functions $\phi_k=(\mu)_k/(2\mu)_k$} \label{SEC:tps}

In this section we will show that the formal power series 
\begin{equation} \label{EQ:secondconcavity}
y(\mu; x) := \sum_{k = 0}^\infty y_k \frac{(\mu)_{k}}{(2 \mu)_k k!} x^k
\end{equation}
is  coefficient-wise log-convex  for any non-negative sequence $\{y_k\}_{k = 0}^\infty$ independent of $\mu$. 

\begin{theorem} \label{THM:logconvexitysechyper}
Suppose $y$ is defined in \eqref{EQ:secondconcavity} with any non-negative sequence $\{y_k\}_{k = 0}^\infty$ independent of $\mu$. Then for any $\alpha,\beta>0$ the generalized Tur\'{a}nian
\begin{equation*}
\Delta_{y}(\alpha, \beta; x)= y(\mu + \alpha; x) y(\mu + \beta; x) - y(\mu; x) y(\mu + \alpha + \beta; x) = \sum_{m = 2}^\infty \bar{\phi}_m x^m
\end{equation*}
has negative power series coefficients $\bar{\phi}_m < 0$, so that the function $\mu \mapsto y(\mu; x)$ is coefficient-wise log-convex for $\mu>0$. 
\end{theorem}

\begin{proof}
By direct computation, we have
\[
\bar{\phi}_m = \sum_{k = 0}^m y_k y_{m - k} \bar{M}_k, 
\]
where 
\begin{equation*}
\bar{M}_k = \frac{1}{k! (m - k)!} \left[\frac{(\mu + \alpha)_k (\mu + \beta)_{m - k}}{(2 (\mu + \alpha))_k (2 (\mu + \beta))_{m - k}} 
 - 
\frac{(\mu)_k (\mu + \alpha + \beta)_{m - k}}{(2 \mu)_k (2 (\mu + \alpha + \beta))_{m - k}} \right].
\end{equation*}
Pairing the terms with indices $k$ and $m-k$, we may write $\bar{\phi}_m$ in the form
\begin{equation} \label{EQ:phi}
\bar{\phi}_m = \sum_{k = 0}^{[m/2]} y_k y_{m - k} \bar{A}_k, 
\end{equation}
where $\bar{A}_k = \bar{M}_k + \bar{M}_{m - k}$ for $k < m/2$ and $\bar{A}_k = \bar{M}_k$ for $k = m/2$. 
We set $\tilde{A}_k = k! (m - k)! \bar{A}_k$. 
Then
\[ 
 \tilde{A}_k =
\begin{cases} 
      \underbrace{\frac{(\mu + \alpha)_{k} (\mu + \beta)_{m - k}}{(2 (\mu + \alpha))_k (2 (\mu + \beta))_{m - k}}}_{u_k} 
      +   \underbrace{\frac{(\mu + \alpha)_{m - k} (\mu + \beta)_{k}}{(2 (\mu + \alpha))_{m - k} (2 (\mu + \beta))_{k}}}_{v_k} \\
      -\underbrace{\frac{(\mu)_{k} (\mu + \alpha + \beta)_{m - k}}{(2 \mu)_k (2 (\mu + \alpha + \beta))_{m - k}}}_{r_k} 
      -   \underbrace{\frac{(\mu)_{m - k} (\mu + \alpha + \beta)_{k}}{(2 \mu)_{m - k} (2 (\mu + \alpha + \beta))_{k}}}_{s_k} \text{ if } \quad k < m/2, \\
      \dfrac{(\mu + \alpha)_{k} (\mu + \beta)_{k}}{(2 (\mu + \alpha))_k (2 (\mu + \beta))_{k}}-  \dfrac{(\mu)_{k} (\mu + \alpha + \beta)_{k}}{(2 \mu)_k (2 (\mu + \alpha + \beta))_{k}}\quad \text{ if } \quad k = m/2. 
   \end{cases}
\]
By~\eqref{EQ:phi}, it suffices to prove that:   
\begin{itemize}
\item [(i)]$\tilde{A}_k \leq 0$ for $k = m/2$, where the equality holds iff $k = 0, 1$; 
\item [(ii)] $\tilde{A}_k < 0$ for $0\le k<m/2$.
\end{itemize}
We first prove (i). If $k = 0, 1$, then it is straightforward to verify that $\tilde{A}_k = 0$. If $k \geq 2$, then $\tilde{A}_k < 0$
is equivalent to 
\begin{equation} \label{EQ:inequality1}
\frac{(\mu + \alpha)_k (\mu + \beta)_k}{(\mu)_k (\mu + \alpha + \beta)_k} < \frac{(2 (\mu + \alpha))_k (2 (\mu + \beta))_k}{(2 \mu)_k (2 (\mu + \alpha + \beta))_k},  
\end{equation}
which is true because 
\begin{multline*}
 \frac{(2 (\mu + \alpha) + i) (2 (\mu + \beta) + i)}{(2 \mu + i) (2 (\mu + \alpha + \beta) + i)} - \frac{(\mu + \alpha + i) (\mu + \beta + i)}{(\mu + i) (\mu + \alpha + \beta + i)} \\
 =\frac{i \alpha \beta (3 i + 4 \mu + 2 \alpha + 4 \beta)}{(\mu + i) (2 \mu + i) (\mu + \alpha + \beta + i) (2 (\mu + \alpha + \beta) + i)} > 0 \quad \text{ for } \quad i > 0. 
\end{multline*}
Next, we prove (ii). In this case, we see that $u_k < s_k$ is equivalent to
\[
\frac{(\mu + \alpha)_k (\mu + \beta)_{m - k}}{(\mu)_{m - k} (\mu + \alpha + \beta)_k} < \frac{(2 (\mu + \alpha))_k (2 (\mu + \beta))_{m - k}}{(2 \mu)_{m - k} (2 (\mu + \alpha + \beta))_k}.
\]
In the light of~\eqref{EQ:inequality1} and $(\mu+\beta)_{m-k}=(\mu+\beta)_{k} (\mu +\beta+k)_{m - 2k}$ for $2k<m$, it suffices to show that 
\[
\frac{(\mu + \beta + k)_{m - 2 k}}{(\mu + k)_{m - 2 k}} < \frac{(2 (\mu + \beta) + k)_{m - 2 k}}{(2 \mu + k)_{m - 2 k}},
\]
which holds true because 
\[
\frac{2 (\mu + \beta) + i}{2 \mu + i} - \frac{\mu + \beta + i}{\mu + i}  = \frac{i \beta}{(\mu + i) (2 \mu + i)} > 0 \quad \text{ for } \quad  i > 0.
\]
Similarly, we can show that the inequalities $v_k < s_k, r_k < s_k$ and $u_k v_k < r_k s_k$ hold for $0\le k<m/2$. Thus, it follows from Lemma~\ref{LEM:positivity} 
that 
$$\tilde{A}_k  = u_k + v_k - r_k - s_k < 0.$$
\end{proof}

Numerical experiments with the "reciprocal" series 
\[
\tilde{y}(\mu; x) := \sum_{k = 0}^\infty y_k \frac{(2 \mu)_{k}}{(\mu)_k k!} x^k
\]
and the related gamma series 
\[
\hat{y}(\mu; x) := \sum_{k = 0}^\infty y_k \frac{\Gamma(\mu+k)}{\Gamma(2\mu+k)k!} x^k
\]
provide strong evidence for the following conjecture.
\begin{conjecture}
Suppose that $\{y_k\}_{k = 0}^\infty$ is a log-concave sequence. Then the functions $\mu\to \tilde{y}(\mu; x)$ and $\mu\to \hat{y}(\mu; x)$  are  coefficient-wise log-concave on $\mu>0$.
\end{conjecture}
So far, we have been unable to prove this claims.

\section{Applications} \label{SEC:app}
In this section we give several examples of concrete special functions whose logarithmic concavity/convexity in parameters can be established using the results of Sections~\ref{SEC:mps},~\ref{SEC:sps}, and~\ref{SEC:tps}.  Our emphasis is on special functions important in fractional calculus. 

\subsection{Generalized hypergeometric function}\label{subsec:gen-hypergeometric}
The generalized hypergeometric function is
defined by the series
\begin{equation}\label{eq:pFqdefined}
{_{p}F_q}\left(\left.\!\!\begin{array}{c} a_1,a_2,\ldots,a_p\\
b_1,b_2,\ldots,b_q\end{array}\right|z\!\right):=\sum\limits_{n=0}^{\infty}\frac{(a_1)_n(a_2)_n\cdots(a_{p})_n}{(b_1)_n(b_2)_n\cdots(b_q)_nn!}z^n.
\end{equation}
The series (\ref{eq:pFqdefined}) converges in the entire complex plane if
$p\leq{q}$ and in the unit disk if $p=q+1$.  In the latter case
its sum can be extended analytically to the whole complex plane
cut along the ray $[1,\infty)$ \cite[Chapter~2]{AAR}. Applications of the previous results to the
generalized hypergeometric functions are mainly based on the
following lemma.
\begin{lemma}\label{lm:HVVKS}
Denote by $e_k(x_1,\ldots,x_q)$  the $k$-th elementary symmetric
polynomial,
$$
e_0(x_1,\ldots,x_q)=1,~~~e_k(x_1,\ldots,x_q)=\!\!\!\!\!\!\!\!\sum\limits_{1\leq{j_1}<{j_2}\cdots<{j_k}\leq{q}}
\!\!\!\!\!\!\!\!x_{j_1}x_{j_2}\cdots{x_{j_k}},~~k\geq{1}.
$$
Suppose that $q\geq{1}$ and $0\leq{r}\leq{q}$ are integers, $a_i>0$,
$i=1,\ldots,q-r$, $b_i>0$, $i=1,\ldots,q$, and
\begin{equation}\label{eq:symmetric-chain1}
\frac{e_q(b_1,\ldots,b_q)}{e_{q-r}(a_1,\ldots,a_{q-r})}\leq
\frac{e_{q-1}(b_1,\ldots,b_q)}{e_{q-r-1}(a_1,\ldots,a_{q-r})}\leq\cdots
\leq\frac{e_{r+1}(b_1,\ldots,b_q)}{e_{1}(a_1,\ldots,a_{q-r})}\leq
e_{r}(b_1,\ldots,b_q).
\end{equation}
Then the sequence of hypergeometric terms \emph{(}if $r=q$ the
numerator is $1$\emph{)},
$$
f_n=\frac{(a_1)_n\cdots(a_{q-r})_n}{(b_1)_n\cdots(b_q)_n},~n=0,1,\ldots,
$$
is log-concave \ie, $0<f_{n-1}f_{n+1}\leq{f_n^2}$, $n=1,2,\ldots$.
The inequality is strict \emph{(}\ie $\{f_n\}_{n\ge0}$ is strictly log-concave\emph{)} unless $r=0$ and $a_i=b_i$ for
$i=1,\ldots,q$.
\end{lemma}
The proof of this lemma for $r=0$ can be found in
\cite[Theorem~4.4]{HVV} and \cite[Lemma~2]{KSJMAA2010}. The latter
reference also explains how to extend the proof to general $r$
(see the last section of \cite{KSJMAA2010}). A simpler sufficient condition for \eqref{eq:symmetric-chain1} and thus for log-concavity of $\{f_{n}\}_{n\ge0}$ is given in the following lemma \cite[Lemma~4]{KKAN2017}.
\begin{lemma}\label{lm:major-elementary}
 Suppose that
 \begin{equation}\label{eq:majorization}
 \sum\nolimits_{j=1}^{k}b_{n_j}\le \sum\nolimits_{j=1}^{k}a_j~\text{for}~k=1,2,\ldots, q-r    
 \end{equation}
 for some $(q-r)$-dimensional sub-vector $(b_{n_1},\ldots,b_{n_{q-r}})$ of $(b_1,\ldots, b_q)$. Then inequalities \eqref{eq:symmetric-chain1} hold true.
\end{lemma}

In view of the obvious fact that reciprocals of the elements of a positive log-concave sequence form a log-convex sequence, an application of Theorems~\ref{THM:logconvexitysecondhyper} and \ref{THM:logconcavitysecondhyper} leads immediately to the following statement.
\begin{theorem}
Let $0\leq{p}\leq{q}$ be integers and suppose that positive parameters $(a_1,\ldots,a_p)$, $(b_1,\ldots,b_q)$ satisfy \eqref{eq:symmetric-chain1} or \eqref{eq:majorization} with $r=q-p$. Then the function 
$$
\mu\to{_{q+2}F_{q+1}}(\mu/2,\mu/2+1/2,b_{1},\ldots,b_{q};1/2,a_{1},\ldots,a_{q};x)
$$
is coefficient-wise log-convex for $\mu>0$, and the function 
$$
\mu\to(\mu)_2\cdot{_{p+2}F_{q+1}}(\mu/2+1,\mu/2+3/2,a_{1},\ldots,a_{p};3/2,b_{1},\ldots,b_{q};x)
$$
with $p\le q$ is coefficient-wise log-concave for $\mu>0$.
\end{theorem}

Theorem~\ref{THM:discretelogconcavity} yields the following statement with $\Delta_f$ defined in \eqref{eq:genTur-defined}.

\begin{theorem}
Let $0\leq{p}\leq{q}$ be integers and suppose that positive parameters $(a_1,\ldots,a_p)$, $(b_1,\ldots,b_q)$ satisfy \eqref{eq:symmetric-chain1} or \eqref{eq:majorization} with $r=p-q$. Then the function
$$
f(\mu;x)={_{p+2}F_{q+1}}(\mu/2,\mu/2+1/2,a_{1},\ldots,a_{p};\mu+1,b_{1},\ldots,b_{q};x)
$$
satisfies the Tur\'{a}n inequality $\Delta_f(\alpha,\beta;x)>0$ for all $\alpha,\beta\in\mathbb{N}$, and the  coefficients at all powers of $x$ in $\Delta_f(\alpha,\beta;x)$ are non-negative if $\alpha+\beta\le4$ .
The function 
$$
\hat{f}(\mu;x)={_{q+2}F_{q+1}}(\mu/2,\mu/2+1/2,b_{1},\ldots,b_{q};\mu+1, a_{1},\ldots,a_{q};x)
$$
satisfies the reverse Tur\'{a}n inequality  $\Delta_{\hat{f}}(\alpha,\beta;x)<0$ for all $\alpha,\beta\in\mathbb{N}$, and the  coefficients at all powers of $x$ in $\Delta_{\hat{f}}(\alpha,\beta;x)$ are non-positive if $\alpha+\beta\le4$ .
\end{theorem}

Finally, Theorem~\ref{THM:logconvexitysechyper} implies that the function 
\[
\mu\to{_{p+1}F_{q+1}}(\mu,a_{1},\ldots,a_{p};2\mu, b_{1},\ldots,b_{q};x)
\]
is coefficient-wise log-convex for $\mu>0$ and any positive parameters $a_i,b_j$.

\subsection{Parameter derivatives of hypergeometric functions}
For arbitrary positive numbers $a,b,c$, define the sequence $\{h_n\}_{n\ge0}$ by 
$$
h_n(a,b,c)=\frac{(\psi(c+n)-\psi(c))(a)_n}{(b)_n},
$$
where $\psi(z)=\Gamma'(z)/\Gamma(z)$.
If $a\geq{b}$, then the sequence  $\{h_n\}_{n\ge0}$ is log-concave since
\begin{multline*}
h_n^2-h_{n-1}h_{n+1}=\frac{(a)_{n-1}(a)_{n}}{(b)_{n-1}(b)_{n}}
\Big\{\frac{a+n-1}{b+n-1}(\psi(c+n)-\psi(c))^2
\\
-\frac{a+n}{b+n}(\psi(c+n-1)-\psi(c))(\psi(c+n+1)-\psi(c))\Big\}>0.
\end{multline*}
The last inequality holds because $y\to{\psi(c+y)-\psi(c)}$ is
concave according to the Gauss formula~\cite[Theorem~1.6.1]{AAR}
$$
(\psi(c+y)-\psi(c))''_y=\psi''(c+y)=-\int\limits_{0}^{\infty}\frac{t^2e^{-t(c+y)}}{1-e^{-t}}dt<0
$$
and hence is log-concave and since $(a+n-1)/(b+n-1)>(a+n)/(b+n)$ for
$a\geq{b}$.  In view of the relations
$$
\frac{\partial}{\partial{a}}\frac{(a)_n}{(b)_n}=h_n(a,b,a),~~~~\frac{\partial}{\partial{b}}\frac{(a)_n}{(b)_n}=-h_n(a,b,b),
$$
these observations combined with Theorem~\ref{THM:logconcavitysecondhyper} lead to the following statement. 
\begin{theorem}
Suppose $a\ge{b}>0$. Then the functions 
$$
\mu\to\frac{\partial}{\partial{a}}{}_3F_2\left(\begin{matrix}
a,\mu/2+1,\mu/2+3/2
\\
b,3/2
\end{matrix};x\right)
$$
$$
\mu\to-\frac{\partial}{\partial{b}}{}_3F_2\left(\begin{matrix}
a,\mu/2+1,\mu/2+3/2
\\
b,3/2
\end{matrix};x\right)
$$
are coefficient-wise log-concave on $(0,\infty)$.
\end{theorem}

In a similar fashion, Theorem~\ref{THM:discretelogconcavity} yields
\begin{theorem}
Suppose $a\ge{b}>0$. Then the functions 
$$
F_1(\mu;x)=\frac{\partial}{\partial{a}}{}_3F_2\left(\begin{matrix}
a,\mu/2,\mu/2+1/2
\\
b,\mu+1
\end{matrix};x\right)
$$
$$
F_2(\mu;x)=-\frac{\partial}{\partial{b}}{}_3F_2\left(\begin{matrix}
a,\mu/2,\mu/2+1/2
\\
b,\mu+1
\end{matrix};x\right)
$$
satisfy the Tur\'{a}n type inequality 
$$
F_i(\mu+\alpha;x)F_i(\mu+\beta;x)-F_i(\mu;x)F_i(\mu+\alpha+\beta;x)\ge0,~i=1,2,
$$ 
 for each $\mu,x>0$ and $\alpha,\beta\in\bN$. Moreover, the coefficients at all powers of $x$ in the above expression are non-negative if, furthermore,  $\alpha+\beta\le4$. 
\end{theorem}

Theorem~\ref{THM:logconvexitysechyper} leads to 
\begin{theorem}
Suppose $a,b,c>0$. Then the functions 
$$
\mu\to\frac{\partial}{\partial{a}}{}_3F_2\left(\begin{matrix}
a,b,\mu
\\
c,2\mu
\end{matrix};x\right),~~~~~~~\mu\to-\frac{\partial}{\partial{b}}{}_3F_2\left(\begin{matrix}
a,b,\mu
\\
c,2\mu
\end{matrix};x\right)
$$
are coefficient-wise log-convex on $(0,\infty)$.
\end{theorem}

\subsection{$k$-Pochhammer and $k$-Gamma series}

In~\cite{DP2007}, the authors introduced for any $k>0$ and $n=0,1,\ldots$,  the $k$-Pochhammer symbol 
\begin{equation} \label{EQ:kPochhammer}
(x)_{n,k}=x(x + k)(x + 2k) \cdots (x + (n -1)k)
\end{equation}
and the $k$-Gamma function
$$
\Gamma_k(x) = \lim\limits_{n\to\infty}\frac{n!k^n(nk)^{x/k-1}}{(x)_{n,k}}=\int_0^{\infty}e^{-t^k/k}t^{x-1}dt.
$$
It is straightforward to verify that
$$
\Gamma_k(x)=k^{x/k-1}\Gamma(x/k),~~(x)_{n,k}=k^n(x/k)_{n}=\frac{\Gamma_k(x+nk)}{\Gamma_k(x)}.
$$
For each integer $r \geq 1$, and $k, \mu > 0$, we set 
\begin{align*}
g_k(\mu; x) & := \sum_{n=0}^{\infty}g_n\frac{(\mu)_{2n,k}}{(2n)!}x^n=\sum_{n=0}^\infty g_n \frac{(\mu/k)_{2n}}{(2n)!}(k^2x)^n,
\\
f_k(\mu; x) &  :=\sum_{n=1}^{\infty}f_n \frac{(\mu)_{2n,k}}{(2n-1)!}x^n=\sum_{n=1}^\infty f_n \frac{(\mu/k)_{2n}}
{(2n-1)!}(k^2x)^n,
\\
y_k(\mu; x) & := \sum_{n=0}^\infty y_n \Gamma_k(\mu+krn)x^n
=k^{\mu/k-1}\sum_{n=0}^\infty y_n\Gamma(\mu/k+rn) [k^{r/k}x]^n, \\
h_k(\mu; x) & := \sum_{n = 0}^\infty \frac{h_n}{(\mu)_{rn,k}}x^n
= \sum_{n=0}^\infty \frac{h_n}{(\mu/k)_{rn}} \Big[\frac{x}{k^r}\Big]^n,
\\
q_k(\mu; x) &:= \sum_{n=0}^\infty q_n\frac{(\mu)_{2n,k}}{(\mu+k)_{n,k}n!} x^n
=\sum_{n=0}^\infty q_n\frac{(\mu/k)_{2n}}{(\mu/k+1)_{n}n!} (kx)^n,
\\
w_k(\mu; x) &:= \sum_{n=0}^\infty w_n\frac{(\mu)_{n,k}}{(2\mu)_{n,k}n!} x^n
=\sum_{n=0}^\infty w_n\frac{(\mu/k)_{n}}{(2\mu/k)_{n}n!} x^n.
\end{align*}
The right hand sides of the above identities imply that we can extend immediately the results of this paper to the $k$-Pochhammer and $k$-Gamma series listed above.  Namely, if $\{ g_n \}_{n = 0}^\infty$ is log-convex and 
 $\{y_n\}_{n = 0}^\infty$,  $\{h_n\}_{n = 0}^\infty$ and $\{w_n\}_{n = 0}^\infty$ are positive, then  $g_k(\mu; x)$, $y_k(\mu; x)$, $h_k(\mu; x)$ and $w_k(\mu; x)$ 
are all coefficient-wise log-convex.  If $\{f_n\}$ and $\{q_n\}$ are log-concave, then $f_k(\mu; x)$ is coefficient-wise log-concave while $q_k(\mu; x)$ is discrete log-concave and coefficient-wise log-concave for integer shifts whose sum does not exceed $4$.  Natural applications of the above series are the so-called $k$-hypergeometric series similar to the classical hypergeometric series with Pochhammer's symbols substituted with $k$-Pochhammer's symbols.  These series can be easily expressed in terms of the classical hypergeometric series as follows:  
$$
{}_pF_{q,k}\left(\!\!\left.\begin{array}{c}a_1,\ldots,a_p\\b_1,\ldots,b_q\end{array}\right|x\right)
={}_pF_{q}\left(\!\!\left.\begin{array}{c}a_1/k,\ldots,a_p/k\\b_1/k,\ldots,b_q/k\end{array}\right|k^{p-q}x\right).
$$
This formula implies that the results of subsection~\ref{subsec:gen-hypergeometric} also hold for the $k$-hypergeometric functions.  

\subsection{The Fox-Wright function}

Given positive vectors $\mathbf{A}=(A_1,\ldots,A_p)$, $\mathbf{B}=(B_1,\ldots,B_q)$ and complex vectors $\a=(a_1,\ldots,a_p)$, $\b=(b_1,\ldots,b_q)$,
the Fox-Wright function (or ''the generalized Wright function'') is defined by the series \cite[(1)]{KPJMAA2020}
\begin{equation}\label{eq:pPsiqdefined}
{_{p}\Psi_q}\left(\left.\!\!\begin{array}{c}(a_1,A_1),\ldots,(a_p,A_p)\\(b_1,B_1),\ldots,(b_q,B_q)\end{array}\right|z\!\right)
={_{p}\Psi_q}\left(\left.\!\!\begin{array}{c}(\a,\mathbf{A})\\(\b,\mathbf{B})\end{array}\right|z\!\right)
=\sum\limits_{n=0}^{\infty}\frac{\Gamma(\mathbf{A} n+\a)}{\Gamma(\mathbf{B} n+\b)}\frac{z^n}{n!},
\end{equation}
where the shorthand notation $\Gamma(\mathbf{A}{n}+\a)=\prod_{j=1}^{p}\Gamma(A_{j}n+a_{j})$ (similarly for $\Gamma(\mathbf{B}{n}+\b)$) has been used.   The series \eqref{eq:pPsiqdefined} has a positive radius of convergence if
$$
\Delta:=\sum\nolimits_{j=1}^{q}B_j-\sum\nolimits_{i=1}^{p}A_i\ge-1.
$$
More precisely, if $\Delta>-1$ the series converges for all finite values of $z$ to an entire function, while for $\Delta=-1$, its radius of convergence equals $\rho$ defined in \eqref{eq:nec1} below, see details in \cite[(3)]{KPJMAA2020}. The above definition implies that for any $\theta>0$
\begin{equation}\label{eq:gFoxWright}
    g(\mu;x):=\sum\limits_{n=0}^{\infty}\frac{\Gamma(\mathbf{A} n+\a)}{\Gamma(\mathbf{B} n+\b)}\frac{(\mu)_{2n}}{(2n)!}\Big(\frac{x}{\theta}\Big)^n=\frac{\sqrt{\pi}}{\Gamma(\mu)}
{_{p+1}\Psi_{q+1}}\left(\left.\!\!\begin{array}{c}(\mu,2),(\a,\mathbf{A})\\(1/2,1),(\b,\mathbf{B})\end{array}\right|\frac{x}{4\theta}\!\right)
\end{equation}
and the condition $\sum B_i\ge\sum A_j$ guarantees that the radius of convergence is positive.  To apply Theorem~\ref{THM:logconvexitysecondhyper} we need conditions ensuring that the sequence 
\begin{equation}\label{eq:Vn}
V(n)=\frac{\Gamma(\mathbf{A} n+\a)}{\theta^n\Gamma(\mathbf{B} n+\b)}
\end{equation}
is log-convex.  As log-convexity is implied by complete monotonicity of the function $V(t)$ on $(0,\infty)$ we can apply conditions for complete monotonicity established in \cite{BCK2021,KPCMFT2016}.  According to \cite[Theorem~4]{KPCMFT2016} and  \cite[Theorem~3.2]{BCK2021}
the function $V(t)$ is logarithmically completely monotonic (and hence completely monotonic) if and only if
\begin{equation}\label{eq:nec1}
\sum\limits_{j=1}^{q}B_j=\sum\limits_{i=1}^{p}A_i,~~~\rho=\prod\limits_{i=1}^{p}A_i^{A_i}\prod\limits_{j=1}^{q}B_j^{-B_j}\le\theta
\end{equation}
and
\begin{equation}\label{eq:Pu-def}
P(u)=\sum\limits_{i=1}^{p}\frac{e^{-a_iu/A_i}}{1-e^{-u/A_i}}-\sum\limits_{i=1}^{q}\frac{e^{-b_iu/B_i}}{1-e^{-u/B_i}}\ge0~~\text{for
all}~u>0.
\end{equation}
Condition \eqref{eq:Pu-def} may be hard to verify. Sufficient conditions in terms of parameters  $\mathbf{A}$, $\mathbf{B}$, $\a,\b$ can be found in \cite[Theorem~5]{KPCMFT2016} and \cite[Theorems~3.4, 3.7, 3.10 and corollaries]{BCK2021}.  For example, \cite[Corollary~3.9]{BCK2021} implies that $\{V(n)\}_{n\ge0}$ is log-convex if \eqref{eq:nec1} holds true and 
\begin{equation}\label{eq:BCKcor3.9}
\sum_{k=1}^{q}B_k\ge \frac{B_j}{b_j-1}\sum_{k=1}^{p}a_k~\text{for}~j=1,\ldots,q.
\end{equation}
On the other hand, choosing $h_{ji}=\omega_i$ in \cite[Theorem~3.7]{BCK2021} we conclude that $\{V(n)\}_{n\ge0}$ is log-convex if together with \eqref{eq:nec1} we have  
\begin{equation}\label{eq:BCKtheor3.7}
A_i\ge\omega_i\sum_{j=1}^{q}B_j,~i=1,\ldots,p~~\text{and}~~b_j/B_j\ge\sum_{i=1}^{p}\omega_ia_i/A_i~,j=1,\ldots,q,
\end{equation}
for some weights $\omega_i\ge0$, $\sum_{i=1}^{p}\omega_i=1$. Another condition is provided by \cite[Theorem~3.10]{BCK2021}. 
Suppose that for some positive integers $\alpha_j$, $\beta_j$, $j=1,\ldots,p$, we have 
$$
A_{k}=1/\alpha_j~\text{for}~k=\alpha_{j-1}+1,\alpha_{j-1}+2,\ldots,\alpha_{j},~j=1,\ldots,p~\text{and}~\alpha_0=0,
$$ 
and, in a similar fashion, 
$$
B_m=1/\beta_j~\text{for}~m=\beta_{j-1}+1,\beta_{j-1}+2,\ldots,\beta_{j},~j=1,\ldots,p~\text{and}~\beta_0=0.
$$
Then for  $a\ge1$ and  $\theta=\prod_{j=1}^{p}(\beta_j/\alpha_j)$ we obtain that 
$$
V(n)=\theta^{-n}\prod_{j=1}^{p}\frac{\Gamma^{\alpha_j}(n/\alpha_j+a)}{\Gamma^{\beta_j}(n/\beta_j+a)}=\theta^{-n}\frac{\prod_{k=1}^{P}\Gamma(A_kn+a)}{\prod_{m=1}^{Q}\Gamma(B_mn+a)},
$$
where $P=\sum_{j=1}^{p}\alpha_j$, $Q=\sum_{j=1}^{p}\beta_j$, is log-convex if 
\begin{equation}\label{eq:BCKtheor3.10}
\begin{split}
0<\alpha_1\le\alpha_2\le\cdots\le\alpha_p, &~~~~0<\beta_1\le\beta_2\le\cdots\le\beta_p
\\
\sum_{j=1}^{k}\alpha_j\le\sum_{j=1}^{k}\beta_j, &~~~k=1,2,\ldots,p.
\end{split}
\end{equation}
We will summarize these facts in the following proposition
\begin{proposition}\label{prop:Vlc}
Suppose any of the following set of conditions holds: \eqref{eq:nec1}$+$\eqref{eq:Pu-def}, \eqref{eq:nec1}$+$\eqref{eq:BCKcor3.9}, \eqref{eq:nec1}$+$\eqref{eq:BCKtheor3.7} or \eqref{eq:BCKtheor3.10} with integer $\alpha_j$, $\beta_j$. Then the sequence $\{V(n)\}_{n\ge0}$ defined in \eqref{eq:Vn} is log-convex, so that $\{[V(n)]^{-1}\}_{n\ge0}$ is log-concave.
\end{proposition}
More sufficient conditions and further details can be found in \cite[Theorem~5]{KPCMFT2016} and \cite[Theorems~3.4, 3.7, 3.10 and corollaries]{BCK2021}. 

Hence, by Theorem~\ref{THM:logconvexitysecondhyper}  we get
\begin{theorem}
 Suppose conditions of Proposition~\ref{prop:Vlc} hold for the sequence $\{V(n)\}_{n\ge0}$ defined in \eqref{eq:Vn}. Then the function $\mu\to g(\mu;x)$ defined in \eqref{eq:gFoxWright} is coefficient-wise log-convex on $[0,\infty)$.
\end{theorem}

Next,  define the functions
\begin{multline}\label{eq:fFoxWright}
    f(\mu;x):=\sum\limits_{n=1}^{\infty}\frac{(\mu)_{2n}}{V(n-1)(2n-1)!}x^n
    =\frac{x\sqrt{\pi}}{8\Gamma(\mu)}\sum\limits_{n=0}^{\infty}\frac{\Gamma(2n+\mu+2)}{V(n)\Gamma(3/2+n)}\frac{(x/4)^n}{n!}
\\
=\frac{x\sqrt{\pi}}{8\Gamma(\mu)}
{_{q+1}\Psi_{p+1}}\left(\left.\!\!\begin{array}{c}(\mu+2,2),(\b,\mathbf{B})\\(3/2,1),(\a,\mathbf{A})\end{array}\right|\frac{\theta{x}}{4}\!\right),
\end{multline}
and
\begin{equation}\label{eq:hFoxWright}
h(\mu;x):=\sum\limits_{n=0}^{\infty}\frac{\Gamma(\mathbf{B} n+\b)}{\Gamma(\mathbf{A} n+\a)}\frac{(\mu)_{2n}}{(\mu+1)_{n}n!}(\theta{x})^n=\mu\,{_{q+1}\Psi_{p+1}}\left(\left.\!\!\begin{array}{c}(\mu,2),(\b,\mathbf{B})\\(\mu+1,1),(\a,\mathbf{A})\end{array}\right|{\theta}x\!\right).
\end{equation}
The condition $\sum_{i=1}^{p}A_i\ge\sum_{i=j}^{q}B_j$ ensures positive radii of convergence for $f$ and $h$. Set

\begin{equation}\label{eq:hhatFoxWright}
\hat{h}(\mu;x):=\sum\limits_{n=0}^{\infty}\frac{\Gamma(\mathbf{A} n+\a)}{\Gamma(\mathbf{B} n+\b)}
\frac{(\mu)_{2n}}{(\mu+1)_{n}n!}\Big(\frac{x}{\theta}\Big)^n=\mu\,{_{p+1}\Psi_{q+1}}\left(\left.\!\!\begin{array}{c}(\mu,2),(\a,\mathbf{A})\\(\mu+1,1),(\b,\mathbf{B})\end{array}\right|\frac{x}{\theta}\!\right).
\end{equation}
The condition $\sum_{i=j}^{q}B_j\ge\sum_{i=1}^{p}A_i$ ensures a positive radius of convergence for $\hat{h}$.

Application of Theorem~\ref{THM:logconcavitysecondhyper} to the function \eqref{eq:fFoxWright} yields:

\begin{theorem} 
Suppose conditions of Proposition~\ref{prop:Vlc} hold for the sequence $\{V(n)\}_{n\ge0}$ defined in \eqref{eq:Vn}. Then the function $\mu\to f(\mu;x)$ defined in \eqref{eq:fFoxWright} is coefficient-wise log-concave on $[0,\infty)$.
\end{theorem}

Application of Theorem~\ref{THM:discretelogconcavity} to the function \eqref{eq:hFoxWright} yields:

\begin{theorem} 
Suppose conditions of Proposition~\ref{prop:Vlc} hold for the sequence $\{V(n)\}_{n\ge0}$ defined in \eqref{eq:Vn}. Then the function $\mu\to h(\mu;x)$ defined in \eqref{eq:hFoxWright} is coefficient-wise log-concave for shifts $\alpha+\beta\le4$ on $[0,\infty)$ and log-concave for the shifts $\alpha,\beta\in\mathbb{N}$; the function $\mu\to\hat{h}(\mu;x)$ defined in \eqref{eq:hhatFoxWright} is coefficient-wise log-convex for shifts $\alpha+\beta\le4$ on $[0,\infty)$ and log-convex for the shifts $\alpha,\beta\in\mathbb{N}$.
\end{theorem}

Finally, note that Theorem~\ref{THM:logconvexitysechyper} implies that 
$$
\mu\to4^{\mu}\Gamma(\mu+1/2)
{_{p+1}\Psi_{q+1}}\left(\left.\!\!\begin{array}{c}(\mu,1),(\a,\mathbf{A})\\(2\mu,1),(\b,\mathbf{B})\end{array}\right|x\!\right)
$$
is coefficient-wise log-convex for all positive values of $\mathbf{A}$, $\mathbf{B}$, $\a$, $\b$.

\paragraph{Acknowledgments.} This work was completed during the stay of D.Karp at the Institute of Mathematics and Informatics of the Bulgarian Academy of Sciences in Sofia, Bulgaria.  The first author expresses his gratitude to the Institute and, especially, to Virginia Kiryakova for hospitality and acknowledges support of PIKOM and Simons programs.  The work of Y.\:Zhang was supported by the NSFC grants No.\ 12101506 and No.\ 12371520, the Natural Science Foundation of the Jiangsu Higher Education Institutions of China No.\ 21KJB110032, and XJTLU Research Development Fund No.\ RDF-20-01-12.

%
\def\cprime{$'$}

\section*{Appendix} \label{SEC:appendix}

\begin{proposition} 
 The sequence $A_0,A_1,\ldots,A_{[m/2]}$ defined by~\eqref{EQ:tildeA} has exactly one change of sign. 
\end{proposition}

\begin{proof}
In the proof Theorem~\ref{THM:logconvexitysecondhyper}, we have shown that: 
(i). $\sum_{k = 0}^{[m/2]} A_k < 0$ for $m \geq 1$; (ii) If $\tilde{A}_k \leq 0$ for some $k < m/2$, then $\tilde{A}_{k - 1} \leq 0$.
Thus, it suffices to prove that  $\tilde{A}_{[m/2]} > 0$ for each $m \geq 2$. Set $k = [m/2]$. 

(a). If $k = m/2 \geq 1$, then $\tilde{A}_k = (\mu + \alpha)_{2 k} (\mu + \beta)_{2 k} - (\mu)_{2 k} (\mu + \alpha + \beta)_{2 k} > 0$ 
because the function $x \mapsto (x + \tau)/x$ is strictly decreasing for any positive $x$ and~$\tau$.

(b). If $k = (m - 1)/2  \geq 1$, \ie, $m = 2 k + 1$, then 
\begin{align*}
\tilde{A}_k & =  \underbrace{(\mu + \alpha)_{2 k} (\mu + \beta)_{2 (k + 1)}}_{a_k} + \underbrace{(\mu + \alpha)_{2 (k + 1)} (\mu + \beta)_{2 k}}_{b_k} \\
      & \quad -\underbrace{(\mu)_{2 k} (\mu + \alpha + \beta)_{2 (k + 1)}}_{c_k} - \underbrace{(\mu)_{2 (k + 1)} (\mu + \alpha + \beta)_{2 k}}_{d_k}.
\end{align*}
We will prove that $\tilde{A}_k > 0$ for $k \geq 1$ by induction. For $k = 1$, we have 
\begin{multline*}
\tilde{A}_1 = \alpha  \beta  \left(\alpha ^3 \beta +2 \alpha ^3 \mu +\alpha ^3+4 \alpha ^2 \beta  \mu +6
   \alpha ^2 \beta +4 \alpha ^2 \mu ^2+12 \alpha ^2 \mu +6 \alpha ^2+\alpha  \beta ^3 \right. \\
  +4
   \alpha  \beta ^2 \mu +6 \alpha  \beta ^2+6 \alpha  \beta  \mu ^2+30 \alpha  \beta  \mu
   +22 \alpha  \beta +4 \alpha  \mu ^3+30 \alpha  \mu ^2+44 \alpha  \mu +17 \alpha \\
   +2 \beta
   ^3 \mu +\beta ^3+4 \beta ^2 \mu ^2+12 \beta ^2 \mu +6 \beta ^2+4 \beta  \mu ^3+30 \beta 
   \mu ^2+44 \beta  \mu +17 \beta +2 \mu ^4 \\
   \left. +20 \mu ^3+44 \mu ^2+34 \mu +12\right) > 0 . 
\end{multline*}
Assume that $\tilde{A}_k = a_k + b_k - c_k - d_k > 0$. Consider 
\begin{align*}
F(\alpha) & := \tilde{A}_{k + 1} \\
& = a_k (\mu + \alpha + 2 k)_2 (\mu + \beta + 2 k + 2)_2 +  b_k (\mu + \alpha + 2 k + 2)_2 (\mu + \beta + 2 k)_2 \\
& \quad -  c_k (\mu + 2 k)_2 (\mu + \alpha + \beta + 2 k + 2)_2 -  d_k (\mu + 2 k + 2)_2 (\mu + \alpha + \beta + 2 k)_2. 
\end{align*}
Then 
\begin{align*}
F(0) & =  a_k (\mu + 2 k)_2 (\mu + \beta + 2 k + 2)_2 +  b_k (\mu + 2 k + 2)_2 (\mu + \beta + 2 k)_2\\
& \quad -  c_k (\mu + 2 k)_2 (\mu + \beta + 2 k + 2)_2 -  d_k (\mu + 2 k + 2)_2 (\mu + \beta + 2 k)_2 \\
& = (a_k + b_k - c_k - d_k)  (\mu + 2 k)_2 (\mu + \beta + 2 k + 2)_2  \\
& \quad + (b_k - d_k) \left[(\mu + 2 k + 2)_2 (\mu + \beta + 2 k)_2  -   (\mu + 2 k)_2 (\mu + \beta + 2 k + 2)_2 \right] \\
& = (a_k + b_k - c_k - d_k)  (\mu + 2 k)_2 (\mu + \beta + 2 k + 2)_2  \\
& \quad + 2 (b_k - d_k) \beta  \left(2 \beta  \mu +3 \beta +8 k^2+4 \beta  k+8 k \mu +12 k+2 \mu ^2+6 \mu
   +3\right) > 0  
\end{align*}
because $b_k > d_k$. Regarding $a_k, b_k, c_k, d_k$ as constants and differentiating with respect to $\alpha$, we have 
\begin{align*}
F'(\alpha) & = a_k (2 \mu + 2 \alpha + 4 k + 1) (\mu + \beta + 2 k + 2)_2 \\
& \quad + b_k (2 \mu + 2 \alpha + 4 k + 5) (\mu + \beta + 2 k)_2 \\
& \quad - c_k (2 \mu + 2 \alpha + 2 \beta + 4 k + 5) (\mu + 2 k)_2 \\
& \quad - d_k (2 \mu + 2 \alpha + 2 \beta + 4 k + 1) (\mu + 2 k + 2)_2 \\
& = (a_k + b_k - c_k - d_k) (2 \mu + 2 \alpha + 2 \beta + 4 k + 5) (\mu + 2 k)_2  \\
& + a_k \left[ (2 \mu + 2 \alpha + 4 k + 1) (\mu + \beta + 2 k + 2)_2 -  (2 \mu + 2 \alpha + 2 \beta + 4 k + 5) (\mu + 2 k)_2\right] \\
& + b_k \left[ (2 \mu + 2 \alpha + 4 k + 5) (\mu + \beta + 2 k)_2 - (2 \mu + 2 \alpha + 2 \beta + 4 k + 5) (\mu + 2 k)_2 \right] \\
& - d_k \left[(2 \mu + 2 \alpha + 2 \beta + 4 k + 1) (\mu + 2 k + 2)_2 - (2 \mu + 2 \alpha + 2 \beta + 4 k + 5) (\mu + 2 k)_2  \right] \\
& > (a_k + b_k - c_k - d_k) (2 \mu + 2 \alpha + 2 \beta + 4 k + 5) (\mu + 2 k)_2 \\
& d_k \left[(2 \mu + 2 \alpha + 4 k + 1) (\mu + \beta + 2 k + 2)_2 + (2 \mu + 2 \alpha + 4 k + 5) (\mu + \beta + 2 k)_2\right. \\
& \left. - (2 \mu + 2 \alpha + 2 \beta + 4 k + 1) (\mu + 2 k + 2)_2 - (2 \mu + 2 \alpha + 2 \beta + 4 k + 5) (\mu + 2 k)_2  \right] \\
& =  (a_k + b_k - c_k - d_k) (2 \mu + 2 \alpha + 2 \beta + 4 k + 5) (\mu + 2 k)_2 \\
& + 2  d_k  \beta  \left(2 \alpha  \beta +4 \alpha  \mu +6 \alpha +2 \beta  \mu +3 \beta +8 k^2+8
   \alpha  k+4 \beta  k+8 k \mu +12 k \right. \\
   & \left.+2 \mu ^2+6 \mu -1\right) > 0.
\end{align*}
The  first inequality above holds since $\min(a_k, b_k) > d_k$, and the second one is true because $k \geq 1$. Thus, we see 
that $\tilde{A}_{k + 1} = F(\alpha) \geq F(0) > 0$. 
\end{proof}

\end{document}